\documentclass[12pt]{amsart}

\usepackage{bm,latexsym,amsfonts,amsmath,amssymb,amsthm,url,graphics,psfrag,amscd,stmaryrd, mathrsfs}
\usepackage[english]{babel}
\usepackage[T1]{fontenc}

\usepackage{graphics}
\usepackage{graphicx}
\usepackage{color}

\newcommand{\R}{\mathbb{R}}

\newcommand{\E}{\mathbb{E}}
\newcommand{\Z}{\mathbb{Z}}

\newcommand{\kD}{\mathcal{D}}
\newcommand{\kA}{\mathcal{A}}
\newcommand{\kB}{\mathcal{B}}
\newcommand{\kC}{\mathcal{C}}

\newcommand{\kO}{\mathcal{O}}

\newcommand{\kG}{\mathcal{G}}

\newcommand{\kM}{\mathcal{M}}

\newcommand{\kH}{\mathcal{H}}

\newcommand{\kU}{\mathcal{U}}

\newcommand{\en}{\mathbb{E} }

\newcommand{\atanh}{\textrm{atanh}}

\newcommand{\lin}{\left[\kern-0.15em\left[}
\newcommand{\rin} {\right]\kern-0.15em\right]}
\newcommand{\linf}{[\kern-0.15em [}
\newcommand{\rinf} {]\kern-0.15em ]}
\newcommand{\ilin}{\left]\kern-0.15em\left]}
\newcommand{\irin} {\right[\kern-0.15em\right[}

\newtheorem{lem}{Lemma}[section]

\newtheorem{prop}[lem]{Proposition}
\newtheorem{theo}[lem]{Theorem}

\newtheorem*{ack}{Acknowledgments}

\title[ Ising model on random regular graphs]
       {\bf Critical behavior of the annealed  Ising model on random regular graphs}

\author{Van Hao Can}
\address{Institute of Mathematics, Vietnam Academy of Science and Technology, 18 Hoang Quoc Viet, 10307 Ha Noi, Viet Nam}
\email{cvhao89@gmail.com}

 \keywords{Ising model; Random graphs; Critical behavior; Annealed measure.
} 
\subjclass[2010]{05C80; 60F5; 82B20}

\begin{document}
\maketitle
\begin{abstract}
In \cite{GGHPb}, the authors have defined an annealed Ising model on random graphs and proved limit theorems for the magnetization of this model on some  random graphs including random 2-regular graphs. Then in \cite{C}, we generalized their results to the class of all random regular graphs. In this paper, we study the critical behavior of this model. In particular, we determine the critical exponents and prove a non standard limit theorem stating that the magnetization scaled by $n^{3/4}$ converges to a specific random variable, with $n$ the number of vertices of random regular graphs. 
\end{abstract}

\section{Introduction}
 Ising model is one of the most well-known model in the field of statistical physics that exhibits  phase transitions. This model has  been investigated fruitfully for integer  lattices, see e.g.\ \cite{G}. Recently, Ising model has been studied in random graphs as a model of the cooperative interaction of spins in random networks, see for instance \cite{AB,DM,DGH,DGM,MS}. As for other   models in random environments, probabilists study this model in   both  quenched setting and  annealed setting. In the quenched one, the Ising model is defined accordingly to typical samples of graphs. On the other hand, in the annealed one, the Ising model is defined by taking information of all realizations of graphs. In contrast of the well-development of studies on quenched setting (see e.g. \cite{DM, DGH, GGHPa, MS}),  there are few contributions in the annealed one. In  two recent papers \cite{GGHPb, DGGHP}, the authors   defined an annealed Ising model as follows.

 \vspace{0.2 cm}
Let $G_n=(V_n,E_n)$ be a random multigraph (i.e. a random graph possibly having self-loops and multiple edges between vertices) with the set of vertices $V_n= \{v_1, \ldots, v_n\}$ and the set of edges $E_n$. A spin $\sigma_i$  is assigned to each vertex $v_i$. Then for any configuration $\sigma \in \Omega_n:= \{+1,-1\}^n$, the Halmintonian is given by 
\[H(\sigma)= -\beta  \sum_{ i \leq j} k_{i,j} \sigma_i \sigma_j - B \sum_{i=1}^n \sigma_i,\]
where $k_{i,j}$ is the number of edges between $v_i$ and $v_j$, where $\beta \geq 0$ is the inverse temperature and  $B \in \R$ is the uniform external magnetic field.  

Then the configuration probability is given by  the {\it annealed measure}: for all $\sigma \in \Omega_n$,
\[\mu_n(\sigma)= \frac{\en(\exp (-H(\sigma))}{\E(Z_n(\beta,B))},\] 
 where $\E$ denotes the expectation with respect to the random graph, and $Z_n(\beta,B)$ is the partition function:
\[Z_n(\beta,B)= \sum\limits_{\sigma \in \Omega_n} \exp( -H(\sigma)).\]
 
\vspace{0.2 cm}
 In \cite{GGHPb}, Giardin\`a,  Giberti,  van der  Hofstad and  Prioriello  study  this annealed  Ising model  on the rank-one inhomogeneous random graph, the random regular graph with degree $2$  and the configuration model with degrees $1$ and $2$.  After determining limits of thermodynamic quantities and the critical inverse temperature, they prove  laws of large numbers  and central limit theorems  for the magnetization. Continuing this work, the authors of \cite{GGHPb} and Dommers investigate  the critical behaviors of the  Ising model on  inhomogeneous random graphs in  \cite{DGGHP}. 
 
\vspace{0.2 cm}
  In \cite{C}, we generalize the result in \cite{GGHPb} for all  random regular graphs, and  show that the thermodynamic limits in quenched and annealed models are actually the same. In this paper, we are going to study  critical behaviors of the annealed model.  More precisely, we aim to determine critical exponents of thermodynamics limits and prove a non-classical scaling limit theorem for the magnetization.
  
  \vspace{0.2 cm}
   Before stating our main results,  we first give some  definitions  following \cite{GGHPb, C} of   the thermodynamic quantities in finite volume.
\begin{itemize}
\item[(i)] The annealed  pressure is given by
\[\psi_n(\beta,B)=\frac{1}{n} \log \E(Z_n(\beta,B)).\]
\item[(ii)] The annealed magnetization is given by
\[M_n(\beta,B)= \frac{\partial}{\partial B} \psi_n(\beta,B).\]
An interpretation of the magnetization is 
\[M_n(\beta,B)=\E_{\mu_n} \left( \frac{S_n}{n}\right),\]
with $S_n$ the total spin, i.e. $S_n = \sigma_1 + \ldots+\sigma_n$. 
\item[(iii)] The annealed susceptibility is given by
\[\chi_n(\beta,B)= \frac{\partial}{\partial B} M_n(\beta,B) =\frac{\partial^2}{\partial B^2} \psi_n(\beta,B).\]
We also have
\[\chi_n(\beta,B)=\textrm{Var}_{\mu_n} \left( \frac{S_n}{\sqrt{n}}\right).\]
\item[(iv)] The annealed specific heat is given by 
\[\kC_n(\beta,B)= \frac{\partial^2}{\partial \beta^2} \psi_n(\beta,B).\]
\end{itemize}   
When the sequence $(M_n(\beta,B))_n$ converges to a limit, say $\kM(\beta,B)$, we define the spontaneous magnetization as $\kM(\beta,0^+)= \lim \limits_{B \searrow 0} \kM(\beta,B)$. Then the critical inverse temperature is defined as 
\[\beta_c = \inf \{\beta>0: \kM(\beta,0^+)>0 \}.\] 
 The uniqueness region of the existence of the limit magnetization is defined as
\[\kU= \{(\beta,B): \beta \geq 0, B \neq 0 \textrm{ or } 0<\beta < \beta_c, B=0\}.\]
 In \cite{C}, we have proved the existence of the limit of thermodynamic quantities.
  \begin{theo} \label{ttd} \cite[Theorem 1.1 ]{C}. Let us consider the Ising model on the random $d$-regular graph with $d\geq 2$. Then  the following assertions hold.
\begin{itemize}
\item[(i)] For all $\beta \geq 0$ and $B \in \R$, the annealed pressure converges 
\begin{eqnarray*}
 \lim_{n\rightarrow \infty} \psi_n(\beta,B) &=&\psi(\beta,B)   = \frac{\beta d}{2}  -B+ \max_{0 \leq t \leq 1} \left[ H_{\beta}(t)    + 2Bt  \right], 
\end{eqnarray*}
where 
\[H_{\beta}(t)=(t-1) \log(1-t)- t \log t + dF_{\beta}(t),\]
with
\begin{equation*} \label{fut}
F_{\beta}(t)= \int \limits_0^{t\wedge(1-t)} \log f_{\beta}(s)ds,
\end{equation*}
and $t\wedge(1-t) = \min \{t, 1-t\}$,
\begin{eqnarray*}
f_{\beta}(s) = \frac{e^{-2\beta}(1-2s)+ \sqrt{1+(e^{-4\beta}-1)(1-2s)^2}}{2(1-s)}.
\end{eqnarray*}
\item[(ii)] For all $(\beta, B) \in \kU$, the magnetization  converges 
\begin{equation*}
\lim_{n\rightarrow \infty}M_n(\beta,B) =\kM(\beta,B)= \frac{\partial}{\partial B} \psi(\beta,B). 
\end{equation*} 
Moreover, the critical inverse temperature is 
\begin{displaymath}
\beta_c= \emph{\atanh}(1/(d-1))=\left \{ \begin{array}{ll}
\frac{1}{2} \log \left( \frac{d}{d-2} \right) & \textrm{ if } \quad  d \geq 3  \\
\infty  & \textrm{ if } \quad  d= 2.  
\end{array} \right.
\end{displaymath}
\item[(iii)] For all $(\beta,B) \in \kU$,  the annealed susceptibility converges 
\[ \lim_{n\rightarrow \infty}\chi_n(\beta,B) = \chi(\beta,B)= \frac{\partial^2}{\partial B^2} \psi(\beta,B) .\]
\end{itemize}   
\end{theo}
The convergence of annealed pressure has been first proved by Dembo, Montanari, Sly and Sun in \cite{DMSS}. By showing the replica symmetry of the partition function, the authors prove that annealed and quenched pressures converge to a common limit, which has been established in \cite{DM}.
Our proof of the convergence of annealed pressure in \cite{C} is based on  the direct relation between the Hamiltonian and the number of disagreeing edeges (i.e. edges with different spins) in  random regular graphs. To characterize the law of the disagreeing edges, we combine the echangeability of the model and many combinatorial computations.   The convergences of magnetization and susceptibility follow from the one of pressure and standard arguments  introduced in \cite{E,GGHPb}.
 
Unfortunately, we are not able to show the convergence of  specific heat, though it is very natural to expect  that $\kC_n(\beta,B)$ tends to the second derivative of $\psi(\beta,B)$ w.r.t $\beta$. Hence,  we  study an ''artificial'' specific heat limit defined as 
\[\kC(\beta,B)=\frac{\partial^2}{\partial \beta^2} \psi(\beta,B).\] 

Following \cite{DGGHP}, we give a definition of critical exponents of  thermodynamic limits. 

\vspace{0.2 cm}
\noindent {\bf Definition}. The annealed critical exponents $\boldsymbol{ \beta, \delta, \gamma, \gamma', \alpha, \alpha'}$ are defined by:
\begin{align*}
 \kM(\beta,0^+) &\asymp (\beta -\beta_c)^{\boldsymbol{\beta}}& \quad & \textrm{for }   \beta \searrow \beta_c,  \\
 \kM(\beta_c,B) &\asymp  B^{\boldsymbol{1/\delta}}& \quad & \textrm{for } B \searrow 0,  \\
  \chi(\beta, 0^+) &\asymp  (\beta_c-\beta)^{-\boldsymbol{\gamma}}& \quad & \textrm{for }  \beta \nearrow \beta_c, \\
  \chi(\beta, 0^+) &\asymp  (\beta-\beta_c)^{-\boldsymbol{\gamma}'}& \quad & \textrm{for }  \beta \searrow \beta_c\\
   \kC(\beta, 0^+) &\asymp  (\beta_c-\beta)^{-\boldsymbol{\alpha}}& \quad & \textrm{for }  \beta \nearrow \beta_c, \\
  \kC(\beta, 0^+) &\asymp  (\beta-\beta_c)^{-\boldsymbol{\alpha}'}& \quad & \textrm{for }  \beta \searrow \beta_c,
\end{align*}
where we write $f(x) \asymp g(x)$ if the ratio $f(x)/g(x)$ is bounded from $0$ and infinity for the specified limit.

\subsection{Main results} Our first result aims at determining  the critical exponents defined above. 
\begin{theo} \label{cexp}(Annealed critical exponents). 
Let us consider the annealed Ising model on random $d$-regular graph with $d\geq 3$. Then the critical exponents satisfy  
\begin{eqnarray*}
\boldsymbol{\beta} &= &\tfrac{1}{2}  \\
\boldsymbol{\delta} &=&3\\
\boldsymbol{\gamma}=\boldsymbol{\gamma '}&=&1 \\
\boldsymbol{\alpha}=\boldsymbol{\alpha'}&=&0. 
\end{eqnarray*}
\end{theo}
In \cite{DGH}, the authors settle the quenched critical exponents for a large class of random graphs, so-called locally-tree like graphs. In particular, for the random regular graphs, the quenched critical exponents satisfy $\boldsymbol{ \beta}=\tfrac{1}{2}$, $\boldsymbol{ \delta}=3$, $\boldsymbol{ \gamma}=1$.  Additionally,  we have proved in \cite{C} that for the case of random regular graphs, the annealed and quenched thermodynamic quantities are equal.  Therefore, the values of $\boldsymbol{ \beta}, \boldsymbol{ \delta},\boldsymbol{ \gamma}$ can be directly deduced from  the result in \cite{DGH}. On the other hand, the values of other critical exponents $\boldsymbol{ \gamma'}, \boldsymbol{ \alpha}, \boldsymbol{ \alpha'}$ are new and are the main contribution of Theorem \ref{cexp}.   

\vspace{0.2 cm}
Our second result is on the asymptotic behavior of the total spin $S_n$ as $n$ tends to infinity. In \cite[Theorem 1.3 and Proposition 1.4]{C}, we have proved that if $(\beta > 0, B \neq 0)$ or $(0<\beta < \beta_c, B=0)$ then   $ S_n$ satisfies a central limit theorem, and if $(\beta > \beta_c, B= 0)$ then $S_n/n$ is concentrated at two opposite values.   In the  following result, we study  the scaling limit  of $S_n$ for the remained case when $\beta = \beta_c$ and $B=0$.  
\begin{theo}  \label{ltc} (Scaling limit theorem at criticality). Consider the annealed Ising model on random $d$-regular graphs with $d\geq 3$. Suppose that $\beta=\beta_c$ and $B=0$. Then 
\[\frac{\sigma_1 + \ldots + \sigma_n }{n^{3/4}} \quad \mathop{\longrightarrow}^{(\kD)} \quad X  \qquad \textrm{ w.r.t. } \mu_n,\]
where $X$ is a random variable with density proportional to 
\[\exp\left(\frac{-(d-1)(d-2)x^4}{12d^2}\right).\] 
\end{theo}
Different from classical central limit theorems, the  scaling limit theorem at criticality  has non-Gaussian limit distribution. This phenomena  has been observed for some spin models, such as for Curie-Weiss model, or Ising model on $\Z^2$ and inhomogeneous random graphs, see \cite{E,ENa,ENb, CNGa, CNGb, DGGHP}. In fact, some authors believe that the critical nature of the total spin has universal scaling limit, see for example \cite{E, DGGHP}. Indeed, they guess that when $\beta= \beta_c $ and $ B=0$,   $S_n$ scaled by $n^{\boldsymbol{\delta}/(\boldsymbol{\delta} +1)}$, with  $\boldsymbol{\delta}$ the exponent of magnetization, converges in law to a random variable whose the tail of the density behaves like  $\exp(-cx^{\boldsymbol{\delta} +1})$ for $x$ large enough. Our results confirm this belief  for the class of random regular graphs. 

\subsection{Discussion} We now make some further remarks on our results.

\vspace{0.2 cm}
 (i) Since $\beta_c$ is finite if and only  if  $d \geq 3$, in our results, we always assume that $d\geq 3$.

\vspace{0.2 cm}
(ii) A simple interpretation of the specific heat is as follows
\[\kC_n(\beta,B)=\textrm{Var}_{\hat{\mu}_n} \left( \frac{ \sum_{ i \leq j} k_{i,j} \sigma_i \sigma_j}{\sqrt{n}}\right),\]
where $\hat{\mu}_n$ is a probability measure on $\kG_n \times \Omega_n$, with $\kG_n$ the sample space of the random $d$-regular graph, given by
\[\int\limits_{\kG_n \times \Omega_n} f d \hat{\mu}_n = \frac{\sum \limits_{\sigma \in \Omega_n} \E \left(f(G,\sigma)e^{-H_n(\sigma)} \right)}{\sum \limits_{\sigma \in \Omega_n} \E \left(e^{-H_n(\sigma)} \right)}. \] 
We notice that $\mu_n$ is a marginal measure of $\hat{\mu}_n$,
$$\mu_n(\sigma)=\E(\hat{\mu}_n(G,\sigma)).$$
Studying the measure  $\hat{\mu}_n$ might give some ideas to derive the convergence of $(\kC_n(\beta,B))$.

\vspace{0.2 cm}
(iii) A natural and interesting question is to generalize our results for the configuration model random graphs with general degree distributions (see \cite{H} for a definition). Comparing with  the case of  random regular graphs, we have additionally a source of randomness coming from the sequence of degrees. This randomness makes the problem  much more difficult. In particular, we  have proved in \cite[Proposition 7.3]{C} that the annealed pressure converges to a limit given by 
\[\psi(\beta,B)= -B + \max_{0\leq t \leq 1} \left[(t-1) \log(1-t)- t \log t+2Bt+ G_{\beta}(t) \right],\]
where $G_{\beta}(t)$ is a Lipschitz function concerning with a large deviation result on the degree distribution of configuration model. Due to the complexity of  $G_{\beta}(t)$,  we are not able to  show the differentiability  of $\psi(\beta,B)$. Without the differentiability, we can not go further to  other  thermodynamic limits or critical exponents. We also remark that when  the degrees of vertices fluctuates,  the authors of \cite{GGHPa} conjecture  that annealed and quenched Ising models behaves differently. In particular, they guess that the  critical inverse  temperatures  are different.   It would be very interesting to know whether the annealed and quenched critical exponents are equal or not. Notice that in the case of inhomogeneous random graphs, though the annealed and quenched  models have different critical inverse temperatures, they have the same critical exponents, see \cite{DGGHP}.

\vspace{0.2 cm}
(iv) On the proof of Theorems \ref{cexp} and \ref{ltc}, we largely use techniques and results in \cite{C,DGGHP}. In particular, to achieve the critical exponents, we exploit the representation of the annealed pressure $\psi(\beta,B)$ in Theorem \ref{ttd} and use Taylor expansion to study the partial derivatives of $\psi$ when variables $\beta, B$ tend to critical values. On the other hand, to prove Theorem \ref{ltc}, we show the convergence of  the generating function of $S_n/n^{3/4}$ as $n$ tends to infinity, by using Laplace method as in \cite{C}. Previously, the same strategy of proof has been applied by the authors in \cite{DGGHP} to identify critical exponents and prove  scaling limit theorems for the case of  inhomogeneous random graphs.

 \vspace{0.2 cm}
 Finally,  the paper is organized as follows. In Section 2, we give a definition of random regular graphs and prove some useful preliminary results. Then,  we prove Theorems \ref{cexp} and \ref{ltc} in Sections 3 and 4 respectively. 

\section{Preliminaries}
\subsection{Random regular graphs} For each $n$, we start with a vertex set $V_n$ of cardinality  $n$ and construct the edge set as follows. For each vertex $v_i$, start with $d$ half-edges incident to $v_i$.  Then we denote by $\kH$ the set of all the half-edges. Select one of them $h_1$ arbitrarily and then choose a half-edge $h_2$ uniformly from $\kH \setminus \{h_1\}$, and match $h_1$ and $h_2$ to form an edge. Next, select arbitrarily another half-edge $h_3$ from $\kH \setminus \{h_1, h_2\}$  and match it to another $h_4$ uniformly chosen from  $\kH \setminus \{h_1, h_2, h_3\}$. Then continue this procedure until there are no more half-edges.  We finally get a multiple random graph that may have self-loops and multiple edges between vertices satisfying all vertices have degree $d$.  We denote the obtained graph by $G_{n,d}$ and call it  {\it random $d$-regular graph}.
\subsection{Preliminary results}
Following the notation in \cite{C}, we denote by $G_{m,1}$  the random 1-regular graph with the vertex set $\bar{V}_m= \{w_1, \ldots, w_m\}$.  For any $k\leq m$,    $X(k,m)$ is the number of edges between $\bar{U}_k= \{w_1, \ldots, w_k\}$ and $\bar{U}^c_k=\bar{V}_m \setminus \bar{U}_k$ in $G_{m,1}$.   Then for all $0 \leq k \leq m$, we  define
\[g_{\beta}( k, m) = \E \left( e^{-2 \beta X(k,m)} \right).\]
We have already proved in \cite[Section 2]{C} that
\begin{equation} \label{mus}
\mu_n(\sigma) = \frac{\E \left(e^{-H_n(\sigma)}\right)}{\E(Z_n(\beta,B))} =\frac{e^{\left(\frac{\beta d}{2} -B \right)n}g_{\beta}( d |\sigma_+|,dn) e^{2B|\sigma_+|}}{\E(Z_n(\beta,B)) },
\end{equation}
where 
\[\sigma_+ = \{v_j: \sigma_j =1\},\]
and 
\begin{equation} \label{ezn}
\E(Z_n(\beta,B))= e^{ \left(\frac{\beta d}{2} - B\right)n} \times \sum_{j=0}^n \binom{n}{j} e^{2Bj} g_{\beta}(dj,dn). 
\end{equation}
In \cite{C}, by deriving  recursive formulas for the number of disagreeing edges $(X(k,m))$, we obtain the following result on the asymptotic behavior of the sequence $(g_{\beta}(dj,dn))$.
\begin{lem} \cite[Lemma 3.1]{C} \label{lgF}
 Suppose that $\beta \geq 0$. Then  there exists a positive constant $C$, such that for all $0 \leq i \leq j \leq n$,
 \begin{equation} \label{mgjj}
  \Big | \left[ \log g_{\beta}(dj,dn)-ndF_{\beta} \left(\frac{j}{n}\right)\right]  - \left[ \log g_{\beta}(di,dn)-ndF_{\beta} \left(\frac{i}{n} \right)\right] \Big| \leq \frac{C(j-i) }{n},
  \end{equation}
  where $F_{\beta}(t)$ is defined in Theorem \ref{ttd}.
\end{lem}

In the  following lemma, we summarize some properties of  critical points of the function $ H_{\beta}(t)    + 2Bt $, which plays a key role in the formula of $\psi(\beta,B)$. 
\begin{lem} \label{lhbt}
  Let $H_{\beta}(t)$ be the function defined in Theorem \ref{ttd} (i). The following statements hold.
\begin{itemize}
\item[(i)] For $\beta \geq 0$ and $B>0$, the equation $\partial_t H_{\beta}(t) + 2B=0$ has a unique solution $t_*=t_*(\beta,B) \in (\tfrac{1}{2}, 1)$.
\item[(ii)] For $\beta > \beta_c$, the equation $\partial_t H_{\beta}(t)=0$ has   a unique solution $t_+=t(\beta) \in (\tfrac{1}{2}, 1)$. Moreover, as $B \searrow 0$, we have $t_* \rightarrow t_+$. 
\item[(iii)] As $\beta \searrow \beta_c$, we have $t_+ \rightarrow \tfrac{1}{2}$.
\item[(iv)] For $\beta < \beta_c$, as $B \searrow 0$, we have $t_* \rightarrow \tfrac{1}{2}$. 
\end{itemize}
\end{lem}
\begin{proof}
Part (i) is proved in Claim $1^*$ in \cite[Section 4]{C}. Parts (ii) and (iv) are Claims  2a and 2b   in \cite[Section 4]{C}. We now prove (iii) by  contradiction.
 Suppose   that $t_+(\beta) $ does not converges to $ \tfrac{1}{2}$ as $\beta \searrow \beta_c$. Then there exist $\varepsilon >0$ and a sequence $(\beta_i) \searrow \beta_c$, such that $|t_+(\beta_i)-\tfrac{1}{2}| \geq \varepsilon$.  We observe that  the sequence $(t_+(\beta_i))$ is bounded in $(\tfrac{1}{2},1)$. Hence there exists a subsequence $(\beta_{i_k}) \searrow \beta_c$, such that the sequence $(t_+(\beta_{i_k}))$ converges to a point $x \in [\tfrac{1}{2},1]$. By the assumption on the value of $(t_+(\beta_i))$, we have $x \geq \tfrac{1}{2} + \varepsilon$. Moreover,  
\begin{displaymath}
\partial_tH_{\beta}(t) =  \log \left( \frac{1-t}{t} \right) +d \partial_tF_{\beta}(t) = \left \{ \begin{array}{ll}
\log \left( \frac{1-t}{t} \right) +\log f_{\beta}(t)  & \textrm{ if } \quad t \in [0, \tfrac{1}{2})   \\
\log \left( \frac{1-t}{t} \right) - \log f_{\beta}(1-t)  & \textrm{ if } \quad t \in (\tfrac{1}{2}, 1].
\end{array} \right.
\end{displaymath}
Since  $\log f_{\beta}(\tfrac{1}{2}) = 0$, we have $\partial_tH_{\beta}(\tfrac{1}{2}^+)=\partial_tH_{\beta}(\tfrac{1}{2}^-)$. Hence, the function $H_{\beta}(\cdot)$ is  differentiable at the point $\tfrac{1}{2}$. In addition, the function $f_{\beta}(t) $  is jointly continuous   at every point  $(t, \beta)$ with $t \leq \tfrac{1}{2}$. Hence,  the function $\partial_tH_{\beta}(t)$ is jointly continuous. Therefore,
\begin{equation*}
0=\lim \limits_{k \rightarrow \infty} \partial_tH_{\beta_{i_k}} \left(t_+(\beta_{i_k})\right)  = \partial_tH_{\beta_c}(x).
\end{equation*}
This leads to a contradiction, since   by Lemma \ref{lbbc} below the equation $\partial_tH_{\beta_c}(t)=0$ has a unique solution $t=\tfrac{1}{2}$.
\end{proof}
The behavior of the function $H_{\beta_c}(t)$ around the extreme point $t=\tfrac{1}{2}$  is described in the following result,  by using Taylor expansion.
 \begin{lem} \label{lbbc}
   Let us consider  $H(t)=H_{\beta_c}(t)$ with $H_{\beta}(t)$ as in Theorem \ref{ttd}. Then we have 
\begin{equation} \label{moh}
\max_{0 \leq t \leq 1} H(t) =H(\tfrac{1}{2}).
\end{equation} 
 Moreover, 
\begin{equation} \label{bph}
H'(\tfrac{1}{2})=H''(\tfrac{1}{2})=H'''(\tfrac{1}{2})=0,
\end{equation} 
 and 
 \begin{equation} \label{hpb}
 H^{(4)} (\tfrac{1}{2}) = \frac{-32 (d-1)(d-2)}{d^2} <0.
 \end{equation}
 \end{lem}
\begin{proof}
Using the same arguments for Claim 2b in \cite[Section 4]{C}, we have $H''(t) \leq 0$ is a consequence of the following
\begin{equation} \label{ibc}
e^{-4 \beta } \big[(d-2)^2(t-t^2)+d-1 \big] \geq t(1-t)(d-2)^2.
\end{equation}
Since $\beta=\beta_c$,
\begin{equation} \label{cdd}
c:=e^{-2 \beta}= \frac{d-2}{d}.
\end{equation}
Hence \eqref{ibc} is equivalent to 
\begin{equation*}
(d-2)^2(t-t^2)+d-1 \geq d^2(t-t^2),
\end{equation*}
or equivalently,
\begin{equation*}
 1 \geq 4 t(1-t)
\end{equation*}
which holds for all $t \in [0,1]$. Hence the function $H(t)$ is  concave. Moreover, by a simple computation we have $H'(\tfrac{1}{2})=0$. Therefore $H(t)$ gets the maximum at $t=\tfrac{1}{2}$.  Now we prove \eqref{bph} and \eqref{hpb}. We observe that 
\begin{equation} \label{hsf}
H(t)=I(t)+dF(t),
\end{equation}
where 
\[I(t)=(t-1) \log (1-t) -t \log t,\]
and $F(t) = F_{\beta_c}(t)$ is defined in Theorem \ref{ttd}. We have 
\begin{eqnarray*}
I'(t)&=&  \log \left( \frac{1-t}{t} \right), \hspace{2 cm} I''(t)= \frac{-1}{t(1-t)}, \\
I'''(t) &=& \frac{1}{t^2} - \frac{1}{(1-t)^2}, \hspace{2 cm} I^{(4)}(t) =  \frac{2}{(t-1)^3} -\frac{2}{t^3}.
\end{eqnarray*}
 Hence 
 \begin{eqnarray} \label{doS}
 I'(\tfrac{1}{2})= I'''(\tfrac{1}{2})= 0, \hspace{0.9 cm} I''(\tfrac{1}{2})= -4, \hspace{0.9 cm} I^{(4)}(\tfrac{1}{2})=-32.
 \end{eqnarray}
On the other hand, 
\begin{displaymath}
F'(t)= \left \{ \begin{array}{ll}
\log f(t)  & \textrm{ if } \quad t \in [0, \tfrac{1}{2})   \\
- \log f(1-t)  & \textrm{ if } \quad t \in (\tfrac{1}{2},1],  
\end{array} \right.
\end{displaymath}
with 
$$f(t)=f_{\beta_c}(t).$$
In addition, $f(\tfrac{1}{2})=1$, so  $F'(\tfrac{1}{2}^+)=F'(\tfrac{1}{2}^-)=0$. Hence $F'(\tfrac{1}{2})=0$ and  $F$ is a $C^{1}$ function on  $(0,1)$. Furthermore,
\begin{displaymath}
F''(t)= \left \{ \begin{array}{ll}
\frac{ f'(t)}{f(t)}  & \textrm{ if } \quad t \in [0, \tfrac{1}{2}) \\
\frac{f'(1-t)}{f(1-t)}  & \textrm{ if } \quad t \in (\tfrac{1}{2},1].  
\end{array} \right.
\end{displaymath}
Therefore, $F''(\tfrac{1}{2}^+)=F''(\tfrac{1}{2}^-)$. Hence, $F''(\tfrac{1}{2})$ exists and   $F$ is a $C^{2}$ function on  $(0,1)$. Similarly,
\begin{displaymath}
F'''(t)= \left \{ \begin{array}{ll}
\frac{ f''(t)f(t) -(f'(t))^2}{f^2(t)}  & \textrm{ if } \quad t \in [0, \tfrac{1}{2})   \\
\frac{  (f'(1-t))^2-f''(1-t)f(1-t)}{f^2(1-t)}  & \textrm{ if } \quad t \in (\tfrac{1}{2}, 1].  
\end{array} \right.
\end{displaymath}
Thus $F'''(\tfrac{1}{2}^+)=F'''(\tfrac{1}{2}^-)=0$, so $F'''(\tfrac{1}{2})=0$ and  $F$ is a $C^{3}$ function on  $(0,1)$. Moreover,
\begin{displaymath}
F^{(4)}(t)= \left \{ \begin{array}{ll}
\frac{ f'''(t)f(t) -f'(t)f''(t)}{f^2(t)} - \frac{2 f'(t) [f''(t)f(t) -(f'(t))^2]}{f^3(t)}  & \textrm{ if } \quad t \in [0, \tfrac{1}{2})   \\
\frac{ f'''(1-t)f(1-t) -f'(1-t)f''(1-t)}{f^2(1-t)} - \frac{2 f'(1-t) [f''(1-t)f(1-t) -(f'(1-t))^2]}{f^3(1-t)}  & \textrm{ if } \quad t \in (\tfrac{1}{2},1]. 
\end{array} \right.
\end{displaymath}
Hence, $F^{(4)}(\tfrac{1}{2}^+)=F^{(4)}(\tfrac{1}{2}^-)$, so $F^{(4)}(\tfrac{1}{2})$ exits and  $F$ is a $C^{4}$ function on  $(0,1)$. We now compute the values $F''(\tfrac{1}{2})$ and $F^{(4)}(\tfrac{1}{2})$. Observe that
\[f(t)= \frac{A(t)}{B(t)},\]
where 
\[A(t)= c(1-2t) + \sqrt{1+(c^2-1)(2t-1)^2} \hspace{1 cm} \textrm{and} \hspace{1 cm} B(t)= 2(1-t),\]
with $c$ as in \eqref{cdd}. Hence 
\begin{eqnarray*}
f'(t)&=& \frac{A'(t)}{B(t)} - \frac{A(t)B'(t)}{B^2(t)}, \\
f''(t)&=& \frac{A''(t)}{B(t)} - \frac{A(t)B''(t) + 2 A'(t) B'(t)}{B^2(t)} + \frac{2 A(t)(B'(t))^2}{B^3(t)},\\
f'''(t)&=& \frac{A'''(t)}{B(t)} - \frac{ A(t)B'''(t) + 3 A'(t) B''(t) +3 A''(t)B'(t)}{B^2(t)} \\
&& + \frac{6 A(t) B'(t) B''(t) + 6 A'(t)(B'(t))^2}{B^3(t)} - \frac{6 A(t)(B'(t))^3}{B^4(t)}.  
\end{eqnarray*}
After some computations, we get
\begin{equation*}
A(\tfrac{1}{2})=1, \hspace{0.8 cm} A'(\tfrac{1}{2})=-2c, \hspace{0.8 cm} A''(\tfrac{1}{2})=4(c^2-1), \hspace{0.8 cm} A'''(\tfrac{1}{2})=0,
\end{equation*}
and 
\begin{equation*}
B(\tfrac{1}{2})=1, \hspace{0.8 cm} B'(\tfrac{1}{2})=-2, \hspace{0.8 cm} B''(\tfrac{1}{2})=B'''(\tfrac{1}{2})=0. 
\end{equation*}
Thus 
\begin{equation*}
f(\tfrac{1}{2})=1, \hspace{0.6 cm} f'(\tfrac{1}{2})=2(1-c), \hspace{0.6 cm} f''(\tfrac{1}{2})=4(1-c)^2, \hspace{0.6 cm} f'''(\tfrac{1}{2})=24(1-c)^2. 
\end{equation*}
Therefore 
\begin{equation} \label{doF}
F'(\tfrac{1}{2})= F'''(\tfrac{1}{2})=0, \hspace{0.4 cm} F''(\tfrac{1}{2})=2(1-c), \hspace{0.4 cm} F^{(4)}(\tfrac{1}{2})=24(1-c)^2 -8(1-c)^3. 
\end{equation}
Combining \eqref{cdd}, \eqref{hsf},  \eqref{doS} and \eqref{doF}, we obtain  desired results.
\end{proof}

\section{Proof of Theorem \ref{cexp}}
We have proved in \cite[Section 4, Claim $1^*$]{C} that for all $\beta \geq 0$ and $B>0$, 
\begin{equation*}
\psi(\beta,B) = \beta d/2 - B + L(t_*,\beta, B),
\end{equation*}
where 
$$L(t,\beta,B) = H_{\beta}(t) +2 Bt,$$
and $t_*=t_*(\beta,B) \in (\tfrac{1}{2},1)$ is the unique zero of the function $\partial_t L(t, \beta, B)$, i.e. 
\begin{equation} \label{hb}
\partial_t L(t_*, \beta, B) =\partial_t H_{\beta}(t_*)+2B=0.
\end{equation}
\subsection{Proof of $\boldsymbol{\delta} =3$}
We have shown in \cite[Section 4]{C} that  for all $\beta \geq 0$ and $B >0$, 
\[\kM(\beta,B)= \frac{\partial}{\partial B} \psi(\beta, B) = -1 +2 t_*,\]
where $t_*$ is the solution of \eqref{hb}.  By Lemma \ref{lhbt} (i) and (iii) we have $t_* \searrow \tfrac{1}{2} $ as $B \searrow 0$. We set 
\[s_*=t_*-\tfrac{1}{2}.\]
Hence 
\begin{equation*} \label{sbo}
s_* \searrow 0 \quad \textrm{as} \quad B \searrow 0.
\end{equation*}
We notice 	also that  for $t > \tfrac{1}{2}$,
\begin{equation} \label{ehp}
\partial_t H_{\beta}(t) = \log \left(\frac{1-t}{t} \right) - d \log f_{\beta}(1-t), 
\end{equation}
with
\[f_{\beta}(1-t)= \frac{e^{-2 \beta} (2t-1) + \sqrt{1+(e^{-4\beta}-1)(2t-1)^2}}{2t}.\]
Therefore the equation \eqref{hb} is equivalent to the following
\begin{equation} \label{mm}
2B=  - \log \theta_1(s_*) + d\log \theta_2(s_*),
\end{equation}
where 
\begin{equation} \label{tms}
\theta_1(s) =\frac{1-2s}{1+2s}, 
\end{equation}
and
\begin{equation} \label{ths}
\theta_2(s) = f_{\beta}(1-s)=  \frac{2e^{-2 \beta}s + \sqrt{1+4(e^{-4\beta}-1)s^2}}{1+2s}.
\end{equation}
We  have 
\begin{equation*}
\theta_1(s_*) -1 = \frac{-4 s_*}{1+2s_*},
\end{equation*}
and 
\begin{equation*}
\theta_2(s_*) -1 = \frac{2s_*(e^{-2 \beta} -1) + \sqrt{1+4(e^{-4\beta}-1)s_*^2} -1}{1+2s_*}.
\end{equation*}
Using  Taylor expansion, we get 
\begin{equation*} \label{tas}
\frac{1}{1+2s_*} = 1-2s_* +4s_*^2 -8 s_*^3 + \kO(s_*^4)
\end{equation*}
and 
\begin{equation*} \label{tac}
\sqrt{1 + 4(e^{-4 \beta} -1) s_*^2} -1 = 2(e^{-4 \beta} -1) s_*^2  + \kO(s_*^4).
\end{equation*}
Therefore
\begin{eqnarray*} 
\theta_1(s_*) -1 &=& -4s_* +8 s_*^2-16s_*^3 + \kO(s_*^4),  \label{tmss}\\
(\theta_1(s_*) -1)^2 &=& 16s_*^2 -64 s_*^3  + \kO(s_*^4), \label{tmssh} \\
(\theta_1(s_*) -1)^3 &=&  -64 s_*^3  + \kO(s_*^4), \label{tmssb}
\end{eqnarray*}
and 
\begin{eqnarray*}
\theta_2(s_*) -1 & = & 2(e^{-2\beta}-1) s_* + 2(e^{-2\beta}-1)^2s_*^2 - 4(e^{-2\beta}-1)^2s_*^3 +\kO(s_*^4), \\
(\theta_2(s_*) -1)^2 &=& 4(e^{-2\beta}-1)^2s_*^2 +8(e^{-2\beta}-1)^3s_*^3 +\kO(s_+^4), \\
(\theta_2(s_*) -1)^3 &=& 8(e^{-2\beta}-1)^3s_*^3 +\kO(s_*^4).
\end{eqnarray*}
Combining these equations and Taylor expansion, we have  
\begin{eqnarray} 
\log \theta_1(s_*) &=& \theta_1(s_*) -1 + \frac{-(\theta_1(s_*) -1)^2}{2} +  \frac{(\theta_1(s_*) -1)^3}{3} + \kO( (\theta_1(s_*) -1)^4) \notag \\
&=& -4 s_* - \frac{16}{3} s_*^3 + \kO( s_*^4), \label{ltm}
\end{eqnarray}
and
\begin{eqnarray}
\log \theta_2(s_*) &=& \theta_2(s_*) -1 + \frac{-(\theta_2(s_*) -1)^2}{2} +  \frac{(\theta_2(s_*) -1)^3}{3} + \kO( s_*^4) \notag \\
&=&2(e^{-2\beta}-1)s_* -\frac{4}{3} (e^{-2\beta}-1)^2(e^{-2\beta}+2) s_*^3 + \kO( s_*^4). \label{lthh} 
\end{eqnarray}
In this subsection, we consider
\[\beta =\beta_c = \frac{1}{2}\log \left(  \frac{d}{d-2} \right) \hspace{1 cm} \textrm{or} \hspace{1 cm}e^{-2 \beta} = \frac{d-2}{d}.\]
Hence 
\begin{eqnarray*}
\log \theta_2(s_*)&=& -\frac{4}{d} s_* + \left( \frac{32}{3d^3} - \frac{16}{d^2} \right) s_*^3 + \kO( s_*^4).  \label{lth}
\end{eqnarray*} 
Therefore,
\begin{eqnarray*} 
-\log \theta_1(s_*) + d \log \theta_2(s_*) &=& \frac{16(d-1)(d-2)}{3d^2}s_*^3+ \kO(s_*^4).  
\end{eqnarray*}
Combining this with \eqref{mm}, we get 
\[ B=  \frac{8(d-1)(d-2)}{3d^2}s_*^3+ \kO(s_*^4). \]
Thus as $B \searrow 0$, 
\[\kM(\beta,B)=2s_* \asymp B^{1/3}.\]
Therefore 
\[\boldsymbol{\delta} =3.\]
\subsection{Proof of $\boldsymbol{\beta} =\tfrac{1}{2}$}
Suppose that $\beta > \beta_c$.  We have proved in    \cite[Claim 2a]{C} that  
\[\kM(\beta,0^+)= -1+2t_+,\]
where $t_+ = t_+(\beta) \in (\tfrac{1}{2},1)$ is the root of  $ \partial_t H_{\beta}(t)$. Moreover, by Lemma \ref{lhbt} (ii) and (iii) we have $t_+ \searrow \tfrac{1}{2}$ as $\beta \searrow \beta_c$. We set 
\[s_+=t_+ -\tfrac{1}{2}.\] 
Then
\begin{equation} \label{kmbe}
\kM(\beta,0^+)= 2s_+,
\end{equation}
and  $s_+$ is the positive solution of the equation $\partial_t H_{\beta} \left(s_++\tfrac{1}{2}\right)=0$. From \eqref{ehp}, we can write this equation
\begin{equation} \label{mdh}
\log \theta_1(s_+)= d \log \theta_2(s_+),
\end{equation}
with $\theta_1(s)$ and $\theta_2(s)$ as in \eqref{tms} and \eqref{ths}. Using similar arguments and calculations for \eqref{ltm} and \eqref{lthh},   we get
\begin{eqnarray*}
\log \theta_1(s_+) &=& -4s_+ -\frac{16}{3} s_+^3 + \kO( s_+^4),
\end{eqnarray*}
and 
\begin{eqnarray*}
\log \theta_2(s_+) &=& 2(e^{-2\beta}-1)s_+ -\frac{4}{3} (e^{-2\beta}-1)^2(e^{-2\beta}+2) s_+^3 + \kO( s_+^4).
\end{eqnarray*}
Hence, for any $\varepsilon >0$, there exists $\delta >0$, such that for all $\beta_c < \beta < \beta_c + \delta$,
\begin{eqnarray} \label{btm}
 -4s_+ - \left( \frac{16}{3} + \varepsilon \right) s_+^3 \leq \log \theta_1(s_+) \leq  -4s_+ - \left( \frac{16}{3} - \varepsilon\right) s_+^3,
\end{eqnarray}
and 
\begin{eqnarray} \label{uth}
\log \theta_2(s_+) \leq  2(e^{-2\beta}-1)s_+ + \left(\varepsilon -\frac{4}{3} (e^{-2\beta}-1)^2(e^{-2\beta}+2)\right) s_+^3,
\end{eqnarray}
and 
\begin{eqnarray} \label{loth}
\log \theta_2(s_+) \geq  2(e^{-2\beta}-1)s_+ - \left(\varepsilon +\frac{4}{3} (e^{-2\beta}-1)^2(e^{-2\beta}+2)\right) s_+^3.
\end{eqnarray}
Using \eqref{mdh}, \eqref{btm} and \eqref{uth}, we get 
\begin{eqnarray*}
 -4s_+ - \left( \frac{16}{3} + \varepsilon \right) s_+^3  \leq 2d(e^{-2\beta}-1)s_+ + d\left(\varepsilon -\frac{4}{3} (e^{-2\beta}-1)^2(e^{-2\beta}+2)\right) s_+^3.
\end{eqnarray*}
Therefore
\begin{eqnarray*}
-4-2d(e^{-2\beta}-1) &\leq& \left( \frac{16}{3} - \frac{4d}{3} (e^{-2\beta}-1)^2(e^{-2\beta}+2) +(d+1) \varepsilon \right)s_+^2. 
\end{eqnarray*}
We observe that as $\beta \searrow \beta_c$,
\[\frac{16}{3} - \frac{4d}{3} (e^{-2\beta}-1)^2(e^{-2\beta}+2) \longrightarrow  \frac{16}{3} - \frac{16}{3}  \frac{(3d-2)}{d^2} = \frac{16(d-1)(d-2)}{3d^2}.   \]
Moreover,
\begin{eqnarray} \label{ebc}
\frac{d-2}{d} - e^{-2\beta} = e^{-2\beta_c} -e^{-2\beta} &=&2e^{-2 \beta_c}(\beta -\beta_c) +\kO((\beta -\beta_c)^2) \notag \\
&=& 2 \left( \frac{d-2}{d}\right)(\beta -\beta_c) +\kO((\beta -\beta_c)^2).
\end{eqnarray} 
Thus
\begin{eqnarray*}
-4-2d(e^{-2\beta}-1) &=& 2d \left(\frac{d-2}{d}  -e^{-2 \beta} \right) =  4(d-2)(\beta - \beta_c)  + o((\beta - \beta_c))
\end{eqnarray*}
Hence for $\beta$ close enough to $\beta_c$,
\begin{eqnarray*}
\big( 4(d-2)- \varepsilon \big)(\beta - \beta_c) \leq \left( \frac{16(d-1)(d-2)}{3d^2} +(d+2) \varepsilon\right)s_+^2.
\end{eqnarray*}
Similarly, using \eqref{mdh}, \eqref{btm} and \eqref{loth} we can also prove that for $\beta$ close to $\beta_c$,
\begin{eqnarray*}
\big( 4(d-2)+ \varepsilon \big)(\beta - \beta_c) \geq \left( \frac{16(d-1)(d-2)}{3d^2} -(d+2) \varepsilon\right)s_+^2.
\end{eqnarray*}
It follows from last two inequalities that 
\begin{equation} \label{spb}
s_+^2 = \frac{3d^2}{4(d-1)} (\beta-\beta_c) +o((\beta-\beta_c)).
\end{equation}
Combining \eqref{kmbe} and \eqref{spb}, we get  
\[\boldsymbol{\beta}=\tfrac{1}{2}.\]
\subsection{Proof of $\boldsymbol{\gamma}=\boldsymbol{\gamma'}=1$} We have shown in \cite[Section 5]{C} that  for $B>0$
\begin{eqnarray} \label{cbb}
\chi(\beta,B)= \frac{-4}{ \partial_{tt} H_{\beta}(t_*)}.  
\end{eqnarray}
In the proof of Claim $1^*$ in \cite[Section 4]{C}, it is  shown that for all $t \in (0,1)$
\begin{equation} \label{ypb}
\partial_{tt}H_{\beta}(t)= \frac{-P(t)}{Q(t)}, 
\end{equation}
where
\[P(t) = d(1-t) \left(e^{-2 \beta}\theta_2(t)-1 \right) +e^{-2 \beta}(2t-1)\theta_2(t)+2-2t,\]
and 
\[Q(t)=t(1-t) \big(e^{-2\beta}(2t-1)\theta_2(t)+2-2t \big),\]
with $$ \theta_2(t) = f_{\beta}(1-t).$$

\vspace{0.2 cm}
\noindent {\bf Case $\beta > \beta_c$}. By Lemma \ref{lhbt} (ii), we have $t_* \rightarrow t_+$ as $B \rightarrow 0^+$, with $t_+ = t_+(\beta)$ the root  of $\partial_t H_{\beta}$.  Therefore, by \eqref{cbb}
\begin{equation} \label{cbh}
\chi(\beta,0^+)= \frac{-4}{ \partial_{tt} H_{\beta}(t_+)}.
\end{equation} 
Using \eqref{ehp}, the equation $\partial_{t}H_{\beta}(t_+)=0$ is equivalent to 
 \[d\log f_{\beta}(1-t_+)= \log \left( \frac{1-t_+}{t_+} \right),\]
or equivalently
\[\theta_2(t_+)= \left( \frac{1-t_+}{t_+} \right)^{1/d} = \left(\frac{1-2s_+}{1+2s_+}\right)^{1/d},\]
where 
$$s_+=t_+ -\tfrac{1}{2}.$$
Notice that  as $\beta \searrow \beta_c$, we have $t_+ \searrow \tfrac{1}{2}$,  thus $s_+\searrow 0$. 
By   Taylor expansion, 
\begin{eqnarray*}
\theta_2(t_+)=\left(\frac{1-2s_+}{1+2s_+}\right)^{1/d}   &=&  1- \frac{4s_+}{d} +\frac{16s_+^2}{d^2} +\kO(s_+^3).
\end{eqnarray*}
Hence 
\begin{eqnarray} \label{eoq}
Q(t_+)=(\tfrac{1}{4} - s_+^2) \big(2e^{-2\beta} \theta_2(t_+)s_+ + 1 - 2s_+ \big) &=& \tfrac{1}{4} + \kO(s_+)  \notag \\
 (\textrm{by } \eqref{spb}) \hspace{1.5cm} &=& \tfrac{1}{4}+ \kO(\sqrt{\beta - \beta_c}).
\end{eqnarray}
Similarly,
\begin{eqnarray*}
P(t_+) &=& d \left(\tfrac{1}{2} -s_+ \right) \left( e^{-2 \beta}  \theta_2(t_+)  -1 \right) +2e^{-2 \beta}\theta_2(t_+)s_+  +1 -2s_+ \\
&=&  d \left( \tfrac{1}{2} -s_+ \right) \left( e^{-2 \beta} \left( 1- \frac{4s_+}{d} +\frac{16s_+^2}{d^2}\right) -1 + \kO(s_+^3) \right) + 1 -2s_+ \\
&& +  2 e^{-2 \beta} \left( 1 - \frac{4s_+}{d} \right) s_+ + \kO(s_+^3) \\
&=&\frac{d}{2} \left( e^{-2 \beta} -\frac{d-2}{d}\right) +d \left( \frac{d-2}{d} -e^{-2\beta}\right)s_+ +4e^{-2\beta}s_+^2
 + \kO(s_+^3). 
\end{eqnarray*}
Combining this with \eqref{ebc} and \eqref{spb}, we get 
\begin{eqnarray} \label{eop}
P(t_+) &=&  \frac{(d-2)(2d+1)}{(d-1)}  (\beta - \beta_c) +o((\beta-\beta_c)).
\end{eqnarray}
It follows from \eqref{ypb}, \eqref{eoq}, \eqref{eop} that 
\begin{equation} \label{htpb}
\partial_{tt} H_{\beta}(t_+) = \frac{-4(d-2)(2d+1)}{(d-1)}  (\beta - \beta_c) +o((\beta-\beta_c)).
\end{equation}
This together with \eqref{cbh} imply  that as $\beta \searrow \beta_c$,
\[\chi(\beta,0^+) \asymp (\beta - \beta_c)^{-1},\]
or equivalently
\[\boldsymbol{\gamma'}=1.\]
\vspace{0.2 cm}
\noindent {\bf Case $\beta < \beta_c$.}   By Lemma \ref{lhbt} (iv), we have  $t_* \rightarrow \tfrac{1}{2}$ as $B \rightarrow 0^+$. Therefore, by \eqref{cbb} 
\begin{equation}
\chi(\beta,0^+)= \frac{-4}{ \partial_{tt} H_{\beta}(\tfrac{1}{2})}.
\end{equation} 
We have $\theta_2(\tfrac{1}{2})=f_{\beta}(\tfrac{1}{2}) =1$. Hence 
\begin{eqnarray*}
Q(\tfrac{1}{2}) = \tfrac{1}{4},
\end{eqnarray*}
and 
\begin{eqnarray*}
P(\tfrac{1}{2}) = \frac{d}{2} (e^{-2\beta} -1) +1 &=& \frac{d}{2}  \left(e^{-2\beta} - \frac{d-2}{d} \right). 
\end{eqnarray*}
Thus
\begin{eqnarray} \label{hbh}
\partial_{tt} H_{\beta}(\tfrac{1}{2}) = - 2d  \left(e^{-2\beta} - \frac{d-2}{d} \right). 
\end{eqnarray}
Using \eqref{ebc} and \eqref{hbh}, we have  as $\beta \nearrow \beta_c$,
\begin{eqnarray*}
\partial_{tt} H_{\beta}(\tfrac{1}{2}) = - 4(d-2) (\beta_c - \beta) + \kO((\beta_c - \beta)^2), 
\end{eqnarray*}
so we get
\[\chi(\beta,0^+) \asymp (\beta_c - \beta)^{-1},\]
and thus
\[\boldsymbol{\gamma}=1.\]
\subsection{Proof of $\boldsymbol{\alpha}=\boldsymbol{\alpha'}=0$ }
We recall  that 
\begin{equation} \label{eots}
\psi(\beta,B)= \beta d/2 - B + L(t_*,\beta, B),
\end{equation}
where $t_*= t_*(\beta, B) \in (\tfrac{1}{2},1)$ is the solution of the following equation
\begin{eqnarray} \label{lbts}
\partial_t L(t_*,\beta, B) = 0.
\end{eqnarray}
In addition, by Claim $1^*$ in \cite{C},  $\partial_{tt} L (t_*,\beta, B) \neq 0$.  Hence, $t_*$ is a differentiable function by the implicit function theorem. Taking derivative in $\beta$ of \eqref{lbts}, we get
\begin{equation*}
\partial_{t \beta} L (t_*, \beta, B) + \partial_{tt} L(t_*, \beta, B)  \partial_{\beta} t_* = 0.
\end{equation*}
Thus 
\begin{equation*}
\partial_{\beta} t_*  = - \frac{\partial_{t \beta} L (t_*, \beta, B)}{ \partial_{tt} L (t_*, \beta, B)}.
\end{equation*}
Using \eqref{eots} and \eqref{lbts}, we get
\begin{eqnarray*}
\partial_{\beta} \psi (\beta, B) = \frac{d}{2} + \partial_t L(t_*,\beta, B) \partial_{\beta} t_* + \partial_{\beta} L(t_*,\beta, B) = \frac{d}{2} + \partial_{\beta} L(t_*,\beta, B). 
\end{eqnarray*}
It follows from  the last two equations that
\begin{eqnarray}
\partial_{\beta \beta}  \psi (\beta, B) &=&  \partial_{\beta \beta} L(t_*,\beta, B) +  \partial_{t \beta } L(t_*,\beta, B) \partial_{\beta }  t_* \notag \\
 &=& \partial_{\beta \beta} L(t_*,\beta, B) - \frac{(\partial_{t\beta } L(t_*,\beta, B))^2}{\partial_{tt} L (t_*,\beta, B)}.  \label{pbbp}
\end{eqnarray}
 For $t > \tfrac{1}{2}$, we have
\begin{eqnarray} \label{patb}
\partial_{t \beta} L (t,\beta, B) = d \partial_{t \beta } F_{\beta} (t) = - d \partial_{\beta} \log f_{\beta}(1-t) = - d p_{\beta}(1-t)
\end{eqnarray}
and 
\begin{eqnarray} \label{pabb}
\partial_{\beta \beta} L (t,\beta, B) = d \partial_{\beta \beta } F_{\beta} (t) =  d \int_0^{1-t} \partial_{\beta} p_{\beta} (s) ds,
\end{eqnarray}
where 
$$p_{\beta}(s) = \frac{\partial_{\beta} f_{\beta}(s)}{f_{\beta}(s)}.$$
{\bf Case  $\beta> \beta_c$}.  By Lemma \ref{lhbt} (ii),  $t_* \rightarrow t_+$ as $B \rightarrow 0^+$. Hence using \eqref{pbbp}, we have 
\begin{eqnarray} \label{pptc}
\partial_{\beta \beta}  \psi (\beta, 0^+) &=& \partial_{\beta \beta} L(t_+, \beta, 0) - \frac{(\partial_{t\beta } L(t_+, \beta, 0))^2}{\partial_{tt} L (t_+, \beta, 0)}. 
\end{eqnarray}
By \eqref{htpb},  as $\beta \searrow \beta_c$
\begin{equation} \label{ptttc}
\partial_{tt} L (t_+, \beta, 0) = \partial_{tt} H_{\beta} (t_+)  =  - \frac{4(d-2)(2d+1)}{d-1} (\beta - \beta_c) +o(\beta - \beta_c).
\end{equation}
By direct calculations,  we can show that
\begin{eqnarray} \label{epbs}
p_{\beta}(s) = \frac{\partial_{\beta} f_{\beta}(s)}{f_{\beta} (s)} = \frac{-2 e^{-2 \beta}(1-2s)}{\sqrt{1+ (e^{-4 \beta} -1)(1-2s)^2}}.
\end{eqnarray}
Using \eqref{spb} and \eqref{ebc}, we get  as $\beta \searrow \beta_c$,
\begin{eqnarray} \label{cebc}
e^{-2 \beta} = \frac{d-2}{d} + \kO(\beta-\beta_c), 
\end{eqnarray}
and 
\begin{eqnarray} \label{cspb}
2t_+ -1 = \sqrt{\frac{3d^2}{d-1}} \sqrt{\beta - \beta_c} + o \left(\sqrt{\beta- \beta_c} \right).
\end{eqnarray}
Using \eqref{patb}, \eqref{epbs}, \eqref{cebc}, \eqref{cspb}  we have as $\beta \searrow \beta_c$,
\begin{eqnarray*}
\partial_{t \beta} L(t_+, \beta, 0) = -d p_{\beta} (1-t_+) = -\frac{2 \sqrt{3}d(d-2)}{\sqrt{d-1}} \sqrt{\beta - \beta_c}  + o \left(\sqrt{\beta - \beta_c} \right). 
\end{eqnarray*}
Hence 
\begin{eqnarray} \label{ptbtc}
(\partial_{t \beta} L(t_+,\beta, 0))^2 =  \frac{12 d^2(d-2)^2}{d-1} (\beta - \beta_c)   + o(\beta - \beta_c). 
\end{eqnarray}
We have 
\begin{equation}
\partial_{\beta} p_{\beta}(s) = \frac{4 e^{-2 \beta} (1-2s)(1-(1-2s)^2) }{\left(\sqrt{1+ (e^{-4 \beta} -1)(1-2s)^2} \right)^3}.
\end{equation}
Therefore  using \eqref{pabb}, we obtain
\begin{eqnarray}
\partial_{\beta \beta} L(t_+, \beta, 0) &=& 4d \int_0^{1-t_+}  \frac{ e^{-2 \beta} (1-2s)(1-(1-2s)^2) }{\left(\sqrt{1+ (e^{-4 \beta} -1)(1-2s)^2} \right)^3} ds \notag \\
&=& 2d \int_{2t_+ -1}^1  \frac{ e^{-2 \beta} u(1-u^2) }{\left(\sqrt{1+ (e^{-4 \beta} -1)u^2} \right)^3} du  \notag \\
&=& 2d \int_{0}^1  \frac{ e^{-2 \beta} u(1-u^2) }{\left(\sqrt{1+ (e^{-4 \beta} -1)u^2} \right)^3} du - 2d \int^{2t_+ -1}_0  \frac{ e^{-2 \beta} u(1-u^2) }{\left(\sqrt{1+ (e^{-4 \beta} -1)u^2} \right)^3} du \notag\\
&=& J_1 - J_2. \label{jmth} 
\end{eqnarray}
We observe that
$$0 \leq \frac{ e^{-2 \beta} u(1-u^2) }{\left(\sqrt{1+ (e^{-4 \beta} -1)u^2} \right)^3} \leq e^{4 \beta}.$$
Hence 
\begin{equation*}
0 \leq J_2 \leq 2de^{4 \beta}(2t_+-1).
\end{equation*}
Combining this inequality with \eqref{cspb} yields that  as $\beta \searrow \beta_c$
\begin{equation*}
J_2=\kO(\sqrt{\beta-\beta_c}). 
\end{equation*}
On the other hand, by using \eqref{cebc}  we have
\begin{equation*}
J_1 =\int_{0}^1  \frac{ 2d^3(d-2) u(1-u^2) }{\left(\sqrt{d^2+ (4-4d)u^2} \right)^3} du + \kO(\beta- \beta_c).
\end{equation*}
Combining the last two equations and \eqref{jmth} gives that
\begin{equation} \label{pbbtc}
\partial_{\beta \beta} L(t_+, \beta, 0) =  \int_{0}^1  \frac{ 2d^3(d-2)  u(1-u^2) }{\left(\sqrt{d^2+ (4-4d)u^2} \right)^3} du + \kO(\beta- \beta_c).
\end{equation}
It follows from  \eqref{pptc}, \eqref{ptttc}, \eqref{ptbtc} and \eqref{pbbtc} that   as $\beta \searrow \beta_c$
\begin{eqnarray*}
\kC(\beta,0^+)=\partial_{\beta \beta} \psi(\beta, 0^+) &=&  \int_{0}^1  \frac{ 2d^3(d-2)  u(1-u^2) }{\left(\sqrt{d^2+ (4-4d)u^2} \right)^3} du + \frac{\frac{12d^2(d-2)^2}{d-1} (\beta - \beta_c) + o(\beta - \beta_c)}{\frac{4(d-2)(2d+1)}{d-1} (\beta - \beta_c) + o(\beta - \beta_c)} \notag \\
&=& \int_{0}^1  \frac{ 2d^3(d-2) u(1-u^2) }{\left(\sqrt{d^2+ (4-4d)u^2} \right)^3} du + \frac{3d^2(d-2)}{2d+1} +o(1).
\end{eqnarray*}
Thus
\[\boldsymbol{\alpha'}=0.\]
{\bf Case $\beta < \beta_c$}. By Lemma \ref{lhbt} (iv), we have $t_* \rightarrow \tfrac{1}{2}$ as $B \rightarrow 0^+$. Therefore, by \eqref{pbbp}
\begin{eqnarray*} \label{pptn}
\partial_{\beta \beta}  \psi (\beta, 0^+) &=& \partial_{\beta \beta} L(\tfrac{1}{2}, \beta, 0) - \frac{(\partial_{t\beta } L(\tfrac{1}{2}, \beta, 0))^2}{\partial_{tt} L (\tfrac{1}{2}, \beta, 0)}. 
\end{eqnarray*}
Using \eqref{hbh}, we obtain
\begin{equation*}
\partial_{tt} L (\tfrac{1}{2}, \beta,0) = \partial_{tt} H_{\beta} (\tfrac{1}{2}) =- 2d  \left(e^{-2\beta} - \frac{d-2}{d} \right). 
\end{equation*}
Moreover,  by \eqref{patb} and \eqref{epbs}
\begin{eqnarray*}
 \partial_{t \beta} L (\tfrac{1}{2}, \beta,0) = -dp_{\beta}(\tfrac{1}{2})=0.
\end{eqnarray*}
We have
\begin{eqnarray*}
 \partial_{\beta \beta} L (\tfrac{1}{2}, \beta,0) &=&  4d \int_0^{1/2}  \frac{ e^{-2 \beta} (1-2s)(1-(1-2s)^2) }{\left(\sqrt{1+ (e^{-4 \beta} -1)(1-2s)^2} \right)^3} ds \notag \\
&=& 2d \int_{0}^1  \frac{ e^{-2 \beta} u(1-u^2) }{\left(\sqrt{1+ (e^{-4 \beta} -1)u^2} \right)^3} du. 
\end{eqnarray*}
Combining the last four equations yields that 
\begin{eqnarray*}
\partial_{\beta \beta} \psi(\beta,0^+) =  2d \int_{0}^1  \frac{ e^{-2 \beta} u(1-u^2) }{\left(\sqrt{1+ (e^{-4 \beta} -1)u^2} \right)^3} du.
\end{eqnarray*}
Now, by using \eqref{cebc}, we get  as $\beta \nearrow \beta_c$, 
\begin{eqnarray*}
\kC(\beta,0^+)= \partial_{\beta \beta} \psi(\beta,0^+)  =  \int_{0}^1  \frac{ 2d^3(d-2)  u(1-u^2) }{\left(\sqrt{d^2+ (4-4d)u^2} \right)^3} du + o(1).
\end{eqnarray*}
Hence 
\[\boldsymbol{\alpha}=0.\]
\section{Proof of Theorem \ref{ltc}}
In this section, we use the same strategy as in the proof of  \cite[Theorem 1.3]{C} to prove our result. In particular, we show that the generating function of $(\sigma_1 + \ldots + \sigma_n)/n^{3/4}$ converges to the one of a specific random variable.  In fact,  Theorem \ref{ltc} is a direct consequence of  the following proposition.
\begin{prop} Suppose that $\beta=\beta_c$ and $B=0$. Then for all $r\in \R$, we have
\begin{equation}
\E_{\mu_n} \left( \exp\left(r \frac{\sigma_1 + \ldots + \sigma_n}{n^{3/4}}\right) \right) \longrightarrow \E\left(e^{rX}\right),
\end{equation}
where $X$ is the random variable defined in Theorem \ref{ltc}.
\end{prop}
\begin{proof} Using \eqref{mus} and \eqref{ezn}, we have 
\begin{eqnarray}
\E_{\mu_n} \left( \exp\left(r \frac{\sigma_1 + \ldots + \sigma_n}{n^{3/4}}\right) \right) &=& \sum _{\sigma \in \Omega_n} \exp \left( \frac{ r \sum \sigma_j}{n^{3/4}}\right) \mu_n(\sigma) \notag \\
& =& \sum _{\sigma \in \Omega_n} \exp \left( \frac{ r \sum \sigma_j}{n^{3/4}}\right)\frac{g( d |\sigma_+|,dn)}{B_n },
\end{eqnarray}
where 
\[\sigma_+ = \{v_j: \sigma_j =1\},\]
and 
\begin{equation} \label{dbn}
B_n=\sum_{j=0}^n \binom{n}{j}g( dj,dn),
\end{equation}
and
\begin{equation*}
g(dj,dn)=g_{\beta_c}(dj,dn)
\end{equation*}
with the sequence $(g_{\beta}(k,m))$ as in Lemma \ref{lgF}. Therefore,
\begin{equation} \label{emun}
\E_{\mu_n} \left( \exp\left(r \frac{\sigma_1 + \ldots + \sigma_n}{n^{3/4}}\right) \right) =\frac{A_n}{B_n},
\end{equation}
with 
\begin{eqnarray}
A_n &=&   \sum _{\sigma \in \Omega_n} \exp \left( \frac{ r \sum \sigma_j}{n^{3/4}}\right) g( d |\sigma_+|,dn) \notag \\
&=&\sum_{j=0}^n \binom{n}{j}\exp \left( \frac{r(2j-n)}{n^{3/4}}\right)g( dj,dn).
\end{eqnarray}
We set
\[j_{*}=[n/2],\]
where $[x]$ stands for the integer part of $x$. Define  for $0\leq j \leq n$,
\begin{eqnarray} \label{xjgj}
x_{j}(n)&= &\binom{n}{j}  g(dj,dn).
\end{eqnarray}
Using the same arguments as in \cite[Section 5]{C}, we  prove in Appendix that
\begin{eqnarray} \label{xjxj}
\frac{x_j(n)}{x_{j_*}(n)} &=& (1+o(1))\sqrt{\frac{j_*(n-j_*)}{j(n-j)}}  \exp \Big(  n\left[H(j/n)-H(j_*/n)\right] \notag\\
&& \qquad \qquad \qquad  + \left[\log g(dj,dn)-ndF(j/n) \right] - \left[\log g(dj_*,dn)-ndF(j_*/n) \right] \Big),
\end{eqnarray}
and 
\begin{equation} \label{tbtm}
\frac{A_n}{B_n} = \frac{\hat{A}_n }{\hat{B}_n } +o(1/n^2), 
\end{equation}
where
\begin{eqnarray*}
\hat{A}_n&=& \sum_{|j-j_*|\leq n^{5/6}} x_j(n) \\ 
\hat{B}_{n}&=& \sum_{|j-j_*|\leq n^{5/6}} \exp\left(\frac{r(2j-n)}{n^{3/4}}\right)x_j(n). \\
\end{eqnarray*}
 Observe that when $|j-j_*|\leq n^{5/6}$,
\begin{equation} \label{mjj}
  \sqrt{\frac{j_*(n-j_*)}{j(n-j)}}  = 1+ \kO(|j-j_*|/n)=1+\kO(n^{-1/6}).
  \end{equation}
Lemma \ref{lgF} implies that   for all $j$,
  \begin{equation} \label{gjj}
  \Big| \big[\log g(dj,dn)-ndF(j/n)\big]  - \big[\log g(dj_*,dn)-ndF(j_*/n)\big]\Big| =\kO(|j-j_*|/n).
  \end{equation}
Using  Taylor expansion  and Lemma \ref{lbbc}, we have
\begin{eqnarray*}
H'(j_*/n) &=& H'(\tfrac{1}{2})  +  H''(\tfrac{1}{2}) \left(\frac{j_*}{n}-\frac{1}{2}\right) + H'''(\tfrac{1}{2}) \left(\frac{j_*}{n}-\frac{1}{2}\right)^2 \\
&& + \kO\left( \left(\frac{j_*}{n}-\frac{1}{2}\right)^3 \right)  = \kO(n^{-3}).
\end{eqnarray*} 
  Similarly,
  \begin{eqnarray*}
  H''(j_*/n) = \kO(n^{-2}), \hspace{1 cm}   H'''(j_*/n) = \kO(n^{-1}), \hspace{1 cm} H^{(4)}(j_*/n)= H^{(4)}(\tfrac{1}{2}) + \kO(n^{-1}).
  \end{eqnarray*}
Hence for all $|j-j_*| \leq n^{5/6}$,
  \begin{eqnarray} \label{hjhj}
  H(j/n)-H(j_*/n) &=& H'\left(\frac{j_*}{n}\right) \left(\frac{j-j_*}{n}\right) + H''\left(\frac{j_*}{n}\right) \frac{(j-j_*)^2}{2n^2} + H'''\left(\frac{j_*}{n}\right) \frac{(j-j_*)^3}{6n^3}  \notag \\
  && + H^{(4)}\left(\frac{j_*}{n}\right) \frac{(j-j_*)^4}{24 n^4} + \kO\left( \left(\frac{j-j_*}{n}  \right)^5\right) \notag \\
   &=& \kO \left(n^{-(3+ 1/6)}\right) + \kO \left(n^{-(2+ 1/3)}\right) +\kO \left(n^{-(1+ 1/2)}\right)   \notag \\
  && + \Big(H^{(4)}\left(\tfrac{1}{2}\right) + \kO(n^{-1})\Big) \frac{(j-j_*)^4}{24 n^4} + \kO(n^{-1/6}) \frac{(j-j_*)^4}{24 n^4}  \notag \\
 & = &  \left(1+\kO(n^{-1/6}) \right)\frac{H^{(4)}(\tfrac{1}{2})}{24}  \frac{(j-j_*)^4}{ n^4} + \kO(n^{-3/2}).
  \end{eqnarray}
   We observe that  by Lemma \ref{lbbc}
  \[\alpha_*:= \frac{H^{(4)}(\tfrac{1}{2})}{24} = -\frac{4(d-1)(d-2)}{3d^2} <0.\] 
  Let $\varepsilon >0$ be any given positive real number.   Using \eqref{xjxj}, \eqref{mjj}, \eqref{gjj} and \eqref{hjhj}, we get that for all $n$ large enough and  $|j-j_*| \leq n^{5/6}$,
  \begin{equation} \label{cdjt}
  x_j(n) \leq (1+\varepsilon)\exp\left((\alpha_*+\varepsilon) \frac{(j-j_*)^4}{n^3}  \right) x_{j_*}(n),
  \end{equation}
  and 
  \begin{equation} \label{cdjv}
  x_j(n) \geq (1-\varepsilon)\exp\left((\alpha_*-\varepsilon) \frac{(j-j_*)^4}{n^3}  \right) x_{j_*}(n).
  \end{equation}
  Using \eqref{cdjt} and similar  arguments as in \cite[Section 5]{C}, we can show that 
  \begin{eqnarray}
  \hat{A}_{n} &=& \sum_{|j-j_*| \leq n^{5/6}} x_j(n) \exp \left(  \frac{2r(j-j_*)}{n^{3/4}} \right) \notag \\
  &\leq &(1+\varepsilon)x_{j_*} (n)\sum_{|k| \leq n^{5/6}}\exp \left( \frac{(\alpha_* + \varepsilon) k^4}{n^3} + \frac{2rk}{n^{3/4}}  \right) \notag \\
  &\leq& (1+\varepsilon)x_{j_*} (n)\sum_{k=-\infty}^{\infty}\exp \left( \frac{(\alpha_* + \varepsilon)k^4}{n^3} + \frac{2rk}{n^{3/4}}  \right) \notag \\
  &\leq& (1+ 2\varepsilon)x_{j_*} (n) n^{3/4}  \int_{-\infty}^{\infty} \exp\left((\alpha_* + \varepsilon) x^4 +2r x \right) dx \notag \\
 (y =2 x) \hspace{1 cm}  &=& (1+ 2\varepsilon) \frac{x_{j_*} (n) n^{3/4}}{2}  \int_{-\infty}^{\infty} \exp \left( \frac{(\alpha_* + \varepsilon) y^4}{16} +r y \right) dy. \label{uoa}
  \end{eqnarray}
   Similarly, using \eqref{cdjv} we have
   \begin{eqnarray} \label{loa}
  \hat{A}_{n} \geq (1- 2\varepsilon) \frac{x_{j_*} (n) n^{3/4}}{2}  \int_{-\infty}^{\infty} \exp \left( \frac{(\alpha_* - \varepsilon) y^4}{16} +r y \right) dy.
   \end{eqnarray}
   Using the same arguments for \eqref{uoa} and \eqref{loa}, we can also prove that
\begin{eqnarray} \label{uob}
\hat{B}_{n}  = \sum_{|j-j_*| \leq n^{5/6}} x_j(n) & \leq&(1+ 2\varepsilon) \frac{x_{j_*} (n) n^{3/4}}{2}  \int_{-\infty}^{\infty} \exp \left( \frac{(\alpha_* + \varepsilon) y^4}{16} \right) dy,
\end{eqnarray}  
and 
\begin{eqnarray} \label{lob}
\hat{B}_{n}   & \geq&(1- 2\varepsilon) \frac{x_{j_*} (n) n^{3/4}}{2}  \int_{-\infty}^{\infty} \exp \left( \frac{(\alpha_* - \varepsilon) y^4}{16} \right) dy.
\end{eqnarray}  
Combining \eqref{uoa}, \eqref{loa}, \eqref{uob} and \eqref{lob}, we obtain
\begin{eqnarray} \label{smh}
\left(\frac{1- 2 \varepsilon}{1+ 2 \varepsilon} \right) \frac{\kA(\alpha_*- \varepsilon, r)}{\kB(\alpha_*+ \varepsilon)} \leq  \frac{\hat{A}_n}{\hat{B}_n} \leq \left(\frac{1+ 2 \varepsilon}{1- 2 \varepsilon} \right) \frac{\kA(\alpha_*+ \varepsilon, r)}{\kB(\alpha_*- \varepsilon)}, 
\end{eqnarray}
where 
\begin{equation*}
\kA(\alpha,r)= \int_{-\infty}^{\infty} \exp \left( \frac{\alpha y^4}{16} +r y \right) dy \hspace{1 cm} \textrm{and} \hspace{1 cm} \kB(\alpha,r)= \int_{-\infty}^{\infty} \exp \left( \frac{\alpha y^4}{16}  \right) dy.
\end{equation*}
 We observe that the derivatives with respect to $\alpha$ at $\alpha_*$  of the functions $\kA(\alpha, r)$ and $\kB(\alpha)$ are bounded. Hence, there exists a constant $C$, such that 
 \begin{equation} \label{acb}
 \Big | \frac{\kA(\alpha_* \pm \varepsilon, r)}{\kB(\alpha_* \pm \varepsilon)} - \frac{\kA(\alpha_*, r)}{\kB(\alpha_*)} \Big | \leq C \varepsilon.
 \end{equation}
 On the other hand, 
 \begin{equation} \label{faab}
 \frac{\kA(\alpha_*, r)}{\kB(\alpha_*)}= \E(e^{rX}),
 \end{equation}
where $X$ is a random variable with density proportional to 
\[\exp \left( \frac{\alpha_*}{16}x^4 \right)= \exp \left( \frac{-(d-1)(d-2)x^4}{12d^2} \right).\]
Combining  \eqref{smh}, \eqref{acb} and \eqref{faab}, and letting   $n$ tends to infinity and $\varepsilon$ tend to $0$, we have
\begin{equation*}
\frac{\hat{A}_n}{\hat{B}_n} \longrightarrow \E(e^{rX}).
\end{equation*}
From this convergence and   \eqref{emun}, we can deduce the desired result.
\end{proof}
\section{Appendix:Proof of \eqref{xjxj} and \eqref{tbtm}}
We will repeat some computations in \cite{C} and use Lemma \ref{lgF} to prove these claims.  
\subsection{Proof of \eqref{xjxj}}
 Using  Stirling's formula, we have 
\begin{equation*} \label{bnj}
\binom{n}{j}= \left(\frac{1}{\sqrt{2 \pi}} + o(1) \right)\sqrt{\frac{n}{j(n-j)}} \exp \left( n I \left(\frac{j}{n}\right) \right),
\end{equation*}
with 
$$I(t)=(t-1) \log (1-t) -t \log t.$$
Therefore using \eqref{xjgj}, we get  
  \begin{eqnarray*} \label{tjj}
  \frac{x_j(n)}{x_{j_*}(n)}&=& (1+o(1)) \sqrt{\frac{j_*(n-j_*)}{j(n-j)} } \exp  \left( n \left[I \left(\frac{j}{n}\right) -  I \left(\frac{j_*}{n}\right)\right] + \log g(dj,dn) - \log g(dj_*,dn) \right)  \notag \\
&=& (1+o(1))\sqrt{\frac{j_*(n-j_*)}{j(n-j)} } \exp \left(  n \left[I \left(\frac{j}{n}\right) -  dF \left(\frac{j}{n}\right)\right] -  n \left[I \left(\frac{j_*}{n}\right) - d F \left(\frac{j_*}{n}\right)\right] \right.  \notag\\
&& \qquad \qquad \qquad \left.   + \left[\log g(dj,dn)-ndF \left(\frac{j}{n}\right) \right] - \left[\log g(dj_*,dn)-ndF \left(\frac{j_*}{n}\right) \right]  \right) \notag \\
&=& (1+o(1))\sqrt{\frac{j_*(n-j_*)}{j(n-j)} } \exp \left(  n \left[H \left(\frac{j}{n}\right) -  H \left(\frac{j_*}{n}\right) \right] \right.  \notag\\
&& \qquad \qquad \qquad \left.   + \left[\log g(dj,dn)-ndF \left(\frac{j}{n}\right) \right] - \left[\log g(dj_*,dn)-ndF \left(\frac{j_*}{n}\right) \right]  \right), \notag 
  \end{eqnarray*}
  which yields \eqref{xjxj}.
  \subsection{Proof of \eqref{tbtm}}  Since $H(t)$ attains the maximum at a unique point $\tfrac{1}{2}$, 
there exists a positive constant $\varepsilon $, such that  for all $\delta \leq \varepsilon$,
\[\max_{|t-\tfrac{1}{2}|\geq \delta} H(t) = \max \{H(\tfrac{1}{2} \pm \delta )\}.\]

\noindent Hence  for  $n$ large enough (such that $n^{-1/6} \leq \varepsilon$), we have for all $|j-j_*| > n^{5/6}$,
 \begin{equation} \label{sach}
 H(j/n)-H(j_*/n) \leq \max \{H(j_*/n \pm n^{-1/6})- H(j_*/n) \}.
\end{equation}  
Using the same arguments for \eqref{hjhj}, we can prove that
\begin{eqnarray*}
H(j_*/n \pm n^{-1/6})- H(j_*/n) =   \alpha_* n^{-2/3} + o(n^{-2/3}). 
\end{eqnarray*}
 Therefore
\begin{equation} \label{nljj}
n \big(H(j_*/n \pm n^{-1/6})- H(j_*/n)\big ) = \alpha_* n^{1/3} +o(n^{1/3}).
\end{equation}
Using $\alpha_*$ and     \eqref{sach}, \eqref{nljj}, we have   for $n$ large enough and $|j-j_*| > n^{5/6}$,  
\begin{equation} \label{nlj}
n \big(H(j/n)-H(j_*/n)\big) \leq   \frac{\alpha_* n^{1/3}}{2}.
\end{equation}
On the other hand, for all $j$
\begin{equation} \label{jsj}
\sqrt{\frac{j_*(n-j_*)}{j(n-j)}}  \leq \sqrt{n}.
\end{equation}
It follows from  \eqref{xjxj},  \eqref{gjj}, \eqref{nlj} and  \eqref{jsj} that  for $n$ large enough and  $|j-j_*| > n^{5/6}$, 
\begin{equation*}
x_j (n) \leq x_{j_*} (n) \sqrt{n} \exp \left(\frac{\alpha_* n^{1/3}}{2} \right)  \leq x_{j_*} (n) n^{-6},
\end{equation*}
 since $\alpha_*<0$.  Therefore
 \begin{equation} \label{ban}
 \bar{A}_n:=\sum_{|j-j_*| > n^{5/6}} x_j(n) \leq x_{j_*}(n) n^{-5},
 \end{equation}
 here we recall that $x_j(n)=0$ for all $j<0$ or $j >n$. Similarly,  for $n$ large enough and  $|j-j_*|> n^{5/6}$,
 \begin{eqnarray*}
 \exp\left(\frac{r(2j-n)}{n^{3/4}}\right)x_j(n) &\leq&  x_{j_*}(n) \sqrt{n} \exp \left(\frac{|r(2j-n)|}{n^{3/4}}\right) \exp \left(\frac{\alpha_*  n^{1/3}}{2}\right) \notag \\
 &\leq & x_{j_*}(n) \sqrt{n} \exp \left(|r| n^{1/4}+ \frac{\alpha_*  n^{1/3}}{2}\right) \notag \\
 & \leq &  x_{j_*}(n) n^{-6}.
 \end{eqnarray*}
 Hence 
 \begin{equation} \label{bbn}
  \bar{B}_n:= \sum_{|j-j_*| > n^{5/6}} \exp\left(\frac{r(2j-n)}{n^{3/4}}\right)x_j(n)  \leq x_{j_*}(n) n^{-5}.
 \end{equation}
Since all the terms $(x_j(n))$ are non negative,
\begin{equation} \label{loan}
\hat{A}_n=\sum_{|j-j_*| \leq n^{5/6}} x_j (n) \geq  x_{j_*} (n),
\end{equation}
and 
\begin{eqnarray} \label{lobn}
\hat{B}_n = \sum_{|j-j_*| \leq n^{5/6}} \exp\left(\frac{r(2j-n)}{n^{3/4}}\right)x_j (n) &\geq&   \exp\left(\frac{r(2j_*-n)}{n^{3/4}}\right)x_{j_*} (n) \notag \\
& \geq &  \exp \left(-\frac{|r|}{n^{3/4}} \right) x_{j_*}(n) \geq \frac{x_{j_*}(n)}{2},
\end{eqnarray}
for $n$ large enough and $r$ fixed.   Finally, combining \eqref{ban}, \eqref{bbn}, \eqref{loan} and \eqref{lobn} yields that
\begin{eqnarray*}
\frac{A_n}{B_n} = \frac{\hat{A}_n  + \bar{A}_n}{\hat{B}_n  + \bar{B}_n} =  \frac{\hat{A}_n}{\hat{B}_n } + o(n^{-2}).
\end{eqnarray*}

 \begin{ack} \emph{  We would like to thank the anonymous referees for their carefully reading and their valuable comments. This work is supported by the Vietnam National Foundation for Science and Technology Development (NAFOSTED) under  Grant number 101.03--2017.01. }
 \end{ack}

\end{document}